\documentclass[a4paper,11pt]{article}

\usepackage[english]{babel}     
\usepackage[latin1]{inputenc}   
\usepackage[T1]{fontenc}

\usepackage{epsfig}

\usepackage{amssymb}
\usepackage{amsmath}            
\usepackage{amsthm}             
\usepackage{amsfonts}           
          
\usepackage{float}              
 
\usepackage{hyperref}

\newcommand{\Z}{\mathbb{Z}}

\newcommand{\Q}{\mathbb{Q}}

\newcommand{\F}{\mathbb{F}}

\newtheoremstyle{th}
{4pt}{5pt}      
{\it}           
{}              
{\bf }          
{:}             
{.5em}          
{}              

\theoremstyle{th} 
\newtheorem{theorem}{Theorem}[section]
\newtheorem{thmIntro}{Theorem}
\newtheorem{LemmaIntro}{Lemma}

\newtheorem{definition}[theorem]{Definition}

\newtheorem{remark}[theorem]{Remark}

\newtheorem{corollary}[theorem]{Corollary}

\newtheorem{proposition}[theorem]{Proposition}

\newtheorem{fact}[theorem]{Fact}

\title{The free group has the dimensional order property}
\date{\today}
\author{Anand Pillay\thanks{Partially supported by NSF grant DMS-1360702} \and Rizos Sklinos\thanks{supported by the 
LABEX MILYON (ANR-10-LABX-0070) of Universit\'{e} de Lyon, within the program "Investissements d'Avenir" (ANR-11-IDEX-0007) operated by the French National Research Agency (ANR).}}

\begin{document}

\maketitle

\begin{abstract} We prove that the common theory $T_{fg}$ of nonabelian free groups has the 
{\em dimensional order property}, or DOP,  implying, for example, that there is no reasonable structure theorem for  $\aleph_{1}$-saturated models of $T_{fg}$. 
\end{abstract}

\section{Introduction} 
By the work of Sela \cite{Sela} and Kharlampovich-Myasnikov \cite{KM},  all nonabelian free groups are elementarily equivalent (as structures in the group language) 
and we denote the common (complete)  first order theory by $T_{fg}$.  For some time we have been suggesting that ``algebraic geometry over the free group'' should be  
the study of the category $Def(T_{fg})$ of definable sets in the free group. 
In a major piece of work \cite{Sela-stability}, Sela proved that  $T_{fg}$ is a {\em stable theory}.  This gives a really new kind of stable theory (or group), and 
there are a host of notions, properties, and invariants, that one can ask about.    The issue has been raised of what groups are definable 
(or more generally interpretable) in a free group  \cite{PPST},\cite{FieldsByronSklinos}. The latter paper succeeds in proving the conjecture 
that no infinite field is definable in the free group. The same paper also shows that centralizers in a free group 
are cyclic groups with no additional induced structure. The latter statement and other results from \cite{FieldsByronSklinos} will be heavily used in the current paper.

A basic invariant of a (complete) first order theory $T$ is the category $Mod(T)$ of models of $T$ (with elementary embeddings) and the focus of much work 
especially in \cite{Shelah-book} was the problem of computing the possible functions $I(\lambda, T)$ = ''number of models, up to isomorphism, of $T$ of cardinality $\lambda$``, 
as $T$ varies, and where possible {\em describing} the models of $T$. 
Now $T_{fg}$ being nonsuperstable has maximum spectrum function, namely $I(\lambda, T) = 2^{\lambda}$ for any $\lambda > \aleph_{0}$.  
But within the class of (countable) stable, nonsuperstable, theories, there is also the possibility of describing or classifying the $\aleph_{1}$-saturated models.  
The {\em dimensional order property} (DOP), which will be formally defined in Section \ref{Stability}, rules out such a classification, and implies that for typical uncountable $\lambda$,  
$T$ has $2^{\lambda}$ $\aleph_{1}$-saturated models of cardinality $\lambda$.  So we prove:

\begin{thmIntro} \label{Maintheorem} 
The common theory of nonabelian free groups has the dimensional order property. 
\end{thmIntro}

It is worth noting that there {\em is} a classification of the $\aleph_{1}$-saturated models of the theory $\mathcal{T}h(\Z,+)$ of the free group on one generator: 
any such model is of the form ${\widehat \Z} \oplus \Q^{(\kappa)}$, where ${\widehat \Z}$ is the profinite completion of $\Z$, and $\kappa \geq \aleph_{1}$.  
In particular for $\kappa > 2^{\aleph_{0}}$ there is a unique $\aleph_{1}$-saturated model of cardinality $\kappa$.  
We had wondered for some time about whether there is a reasonable description of the $\aleph_{1}$-saturated models of the common theory 
of nonabelian free groups, and our main theorem implies a negative answer.

We will prove that $T_{fg}$ has the DOP by showing:

\begin{LemmaIntro} \label{MainLemma}
Let $a,b\in \F$ be independent generics. Then their centralizers 
$C(a)$, $C(b)$ are orthogonal. 
\end{LemmaIntro}
 
It is well-known that Lemma \ref{MainLemma} implies the DOP but we nevertheless give details of the reduction (and more) in Section \ref{Stability}. 

The proof of Lemma \ref{MainLemma} makes use of recent results in \cite{FieldsByronSklinos}. In fact there is a further reduction, using some geometric stability,  
to proving that the centralizers $C(a)$, $C(b)$ are not definably isomorphic, and the latter statement is what is 
actually  proved in Section \ref{CyclicTowers}. 

The question of whether $T_{fg}$ has the DOP was raised by the first author when  the second author was his Ph.D. student in Leeds. 

In the remainder of the introduction we recall some general facts about the model theory of the free group. In Section \ref{Stability} we discuss stability and the DOP property. 
In Section \ref{CyclicTowers} the main technical result is proved and we will introduce there the required machinery. 

We will assume some familiarity with model theory and stability, although in Section \ref{Stability} we will recall details of some classification-theoretic notions. 

As above, $T_{fg}$ denotes the common theory of nonabelian free groups, and is complete. We typically let $\F$ denote a standard model, 
namely a free group $\F_{k}$ of rank $k$ with $k\geq 2$.  Since $T_{fg}$ is stable, stable group theory applies. 
We let $\bar M$ be a saturated model.  We recall some facts and results concerning $T_{fg}$.

\begin{fact} \label{AllinOne}
\
\begin{itemize}
 \item[(i)] The theory $T_{fg}$ is connected.
 \item[(ii)] Let $p_{0}$ denote the unique generic type of the free group over $\emptyset$. Then $p_{0} = tp(a/\emptyset)$ when $a$ is any primitive element of $\F$ (i.e. member of a free basis).
Moreover, if $a_{1},...,a_{n}\in \F$ extends to a basis of $\F$ then $a_{1},..,a_{n}$ are independent (in the sense of nonforking) realizations of $p_{0}$. 
 \item[(iii)] The proper nontrivial definable (with parameters) subgroups of $\F$ are precisely the cyclic subgroups, and hence are the finite index subgroups of centralizers.
 \item[(iv)] Let $a\in \F\setminus\{1\}$. Then $(C(a),.)$ as a  subgroup of the saturated model, is ``stably embedded" in the sense that any set of $n$-tuples 
from $C(a)$ which is definable (with parameters) in the ambient structure $\bar M$ is definable (with parameters) in the structure  $(C(a),.)$. 
In particular $C(a)$ has $U$-rank $1$ and is locally modular (or $1$-based), as a definable group in $\bar M$. 
\end{itemize}
\end{fact}

Facts (i) and (ii) appear in \cite{Pillay-forking} (although a quick proof follows from results of Poizat \cite{Poizat} on $\F_{\omega}$),  
Fact (iii) is Theorem 3 of  \cite{PPST}, and Fact (iv) is Corollary 6.27 of \cite{FieldsByronSklinos}.

\section{Stability}\label{Stability}
The aim of this section is to give more precise details about the DOP and  how we will plan to prove it in the case at hand.

The book \cite{Shelah-book} is the basic reference for classification theory and associated notions. But we  also refer the reader to 
Section 4 of Chapter 1 of \cite{Pillay-book}, where there is an account of notions such as $a$-model, domination, weight, etc., which we summarize here. 

The capital letter $T$ will denote a countable, complete, stable theory, and we often work in a saturated model $\bar M$.  

What we call (following Makkai \cite{Makkai-basic}) an $a$-model in \cite{Pillay-book} is what Shelah 
calls  an $F^{a}_{\kappa (T)}$-saturated model. Recall that $\kappa(T)$ is the smallest infinite cardinal $\kappa$, for which there is no chain $\{p_{\alpha}(x)\in S(A) \ | \ \alpha<\kappa\}$ 
of complete types such that for all $\alpha<\beta<\kappa$, $p_{\beta}$ is a nonforking extension of $p_{\alpha}$. When $T$ is superstable 
$\kappa (T) = \aleph_{0}$, and an $a$-model is just a model in which all strong types over finite sets are realized. When $T$ 
is nonsuperstable, as is the case of $T_{fg}$, $\kappa (T) = \aleph_{1}$, and an $a$-model is precisely a  model in which all strong types over countable sets are realized, 
which just amounts to being $\aleph_{1}$-saturated.  

If $p$ over $A$ and $q$ over $B$ are stationary types, they are said to be {\em  orthogonal} if for any $C\supseteq A\cup B$, the 
nonforking extensions $p|C$ of $p$ and $q|C$ of $q$ over $C$, are almost orthogonal in the sense that if $a$ realizes $p|C$ and $b$ realizes $q|C$ then 
$a$ is independent from $b$ over $C$. Note that  almost orthogonality for the stationary types $p|C$, $q|C$ is equivalent to saying that 
$(p|C)(x) \cup( q|C)(y)$ determines a complete type $r(x,y)$ over $C$.  

There is also the notion of {\em orthogonality to a set}. The stationary type $p(x)\in S(A)$ is said to be orthogonal to a set $B$ if $p$ is orthogonal 
to every strong type over $B$. This is also characterized as follows:  Let $A'$ realize $stp(A/B)$ such that $A'$ is independent from $A$ over $B$. 
Then $p$ is orthogonal to the copy $p'$ of $p$ over $A'$. 

There are prime models in the category of $a$-models, and the corresponding notion of isolation is closely related to domination. 
For example, suppose $M$ is an $a$-model, $c$ a tuple, and $M_{1}$ is $a$-prime over $Mc$. Then $c$ {\em dominates} $M_{1}$ over $M$, 
namely whenever $c$ is independent from a set $B$ over $M$, then $M_{1}$ is independent from $B$ over $M$. 

\begin{definition}\label{DefDOP} 
$T$ has the dimensional order property, or DOP, if there are $a$-models $M_{0}, M_{1}, M_{2}, M_{3}$ and $p(x)\in S(M_{3})$ such that:
\begin{itemize}
 \item[(i)] The model $M_{0}$ is contained in $M_1$ and in $M_2$, moreover $M_{1}$ is independent from $M_{2}$ over $M_{0}$.
 \item[(ii)] The model $M_{3}$ is $a$-prime over $M_{1}\cup M_{2}$. 
 \item[(iii)] The type $p$ is orthogonal to both $M_{1}$ and $M_{2}$. 
\end{itemize}

\end{definition}

For superstable $T$ the DOP is a Shelah-style dividing line for $a$-models, in the sense that assuming DOP gives a nonstructure theorem (many $a$- models) 
and assuming NDOP one obtains a structure theorem (any $a$-model is $a$-prime over a suitable tree of small models).  
This leads to the so-called Main Gap for $a$-models of superstable theories, see \cite{Shelah-book}, \cite{Harrington-Makkai}. 
For stable, nonsuperstable theories, we have the nonstructure theorem \cite{Shelah-book}:

\begin{fact} 
Suppose $T$ is nonsuperstable and has DOP. Then for any uncountable $\lambda$ such that $\lambda = \lambda^{\omega}$, $T$ has $2^{\lambda}$ $\aleph_{1}$-saturated models of cardinality $\lambda$. 
\end{fact}

However there is in general no nice structure theorem for $\aleph_{1}$-saturated models of nonsuperstable theories with NDOP, and the 
``Main Gap"  for $\aleph_{1}$-saturated models remains open. 

We now aim towards reducing  the proof of the main theorem in the introduction to a concrete nondefinability statement about a standard model $\F$.

\begin{proposition}\label{PropReduction}
Let $\{e_{1},e_{2}\}$ be part of a basis of $\F$. Assume that the unique isomorphism between the cyclic subgroups 
$\langle e_{1}\rangle$ and $\langle e_{2} \rangle$ (mapping $e_{1}^{n}$ to $e_{2}^{n}$, for $n\in \Z$) is not definable in $\F$.
Then $T_{fg}$ has the DOP. 
\end{proposition} 
\begin{proof}
We consider $\F$ as an elementary substructure of a very saturated model $\bar M$. 
Let $q_{1}$ be the generic type of $C(e_{1})^{0}$ (the connected component of $C(e_1)$), which we note that it is a stationary type over $e_{1}$. 
Likewise let $q_{2}$ be the generic type of $C(e_2)^0$. 
\\ \\
{\em Claim 1.} The stationary types $q_{1}, q_{2}$ are orthogonal. 
\\
{\em Proof of Claim 1.}  Suppose otherwise. So for some model $M$ containing the data, there are $a$ realizing $q_{1}|M$ and $b$ realizing $q_{2}|M$, 
such that $a$ forks with $b$ over $M$. As both $q_{1}$, $q_{2}$ have $U$-rank $1$, $a$ and $b$ are interalgebraic over $M$, and $U(tp(a,b/M)) = 1$. 
But using Fact \ref{AllinOne}(iv), the group $C(e_{1})\times C(e_{2})$ is $1$-based  hence by \cite[Proposition 4.5, p.170]{Pillay-book},  
$tp(a,b/M)$ is the generic type of a coset of a connected type-definable (over $acl(e_1,e_2)$) subgroup $H$ of $C(e_{1})\times C(e_{2})$. Thus, $H$ itself has $U$-rank $1$. 
As $H\leq C(e_{1})^{0}\times C(e_{2})^{0}$ which is torsion-free divisible, it is clear that $H$ is the graph of an isomorphism between 
$C(e_{1})^{0}$ and $C(e_{2})^{0}$ defined over $acl(e_{1},e_{2})$.  By compactness, there are definable finite index subgroups 
$G_{1}$ of $C(e_{1})$ and $G_{2}$ of $C(e_{2})$ and a definable isomorphism $f$ between $G_{1}$ and $G_{2}$ with everything defined over $\F$.  
Looking at points in the model $\F$,  $f$ restricts to an isomorphism between $G_{1}(\F)$ and $G_{2}(\F)$ which we still call $f$. 
But $G_{1}(\F)$, being a finite index subgroup of the cyclic group $\langle e_{1}\rangle$ is precisely $\langle e_{1}^{k}\rangle$ for some $k>0$, 
and likewise $G_{2}(\F)$ is $\langle e_{2}^{\ell}\rangle$ for some $\ell > 0$, and $f$ takes $e_{1}^{k}$ to $e_{2}^{\ell}$. 

By precomposing with the isomorphism between $\langle e_{1} \rangle$ and $\langle e_{1}^{k} \rangle$ obtained by raising to the kth power, 
and postcomposing with the inverse of the analogous isomorphism between $\langle e_{2}\rangle$ and $\langle e_{2}^{\ell}\rangle$, 
gives an isomorphism between $\langle e_{1}\rangle$ and $\langle e_{2}\rangle$ definable in $\F$, contradicting the assumption of the proposition. Thus, the claim is proved.

Now let $M_{0}$ be an $a$-model independent from $e_{1}, e_{2}$. Then $e_{1}$ is independent from $e_{2}$ over $M_{0}$. 
Let $M_{1}$ be $a$-prime over $(M_{0}, e_{1})$ and $M_{2}$ $a$-prime over $(M_{0}, e_{2})$. Since $e_1$ is independent from $e_2$ over $M_0$, it follows that 
$M_1$ is independent from $M_2$ over $M_0$. Finally let $M_{3}$ be $a$-prime over $M_{1}\cup M_{2}$.  
Let $c=e_{1}\cdot e_{2}$, so $c\in M_{3}$. Let $q_{c}$ be the generic type (over $c$) of $C(c)^{0}$, 
and $r_{c}$ its nonforking extension over $M_{3}$. We will show that the tuple $(M_0, M_1, M_2, M_3, r_c)$ satisfies 
the conditions of Definition \ref{DefDOP}. By construction and the discussion above we only need to check condition (iii).  
\\ \\
{\em Claim 2.} The stationary type $r_{c}$ is orthogonal to each of $M_{1}$, $M_{2}$. 
\\
{\em Proof of Claim 2.}  We will just show that $r_c$ is orthogonal to $M_{1}$. Let $\alpha$ be an automorphism of $\bar M$ fixing $M_{1}$ pointwise, such that 
$M_{3}' = \alpha(M_{3})$ is independent with $M_{3}$ over $M_{1}$. Let $c' = \alpha(c)$, and $q_{c'}$ (over $c'$), $r_{c'}$ (over $M_{3}'$) 
be the copies under $\alpha$ of $q_{c}$ and $r_{c}$ respectively. 

So by the earlier characterization of orthogonality to a set,   we have to show that $r_{c}$ is orthogonal to $r_{c'}$. As $r_{c}$ is 
the unique nonforking extension over $M_{3}$  of $q_{c}$, and $r_{c'}$ the unique nonforking extension of $q_{c'}$ over $M_{3}'$, 
this is equivalent to showing that $q_{c}$ and $q_{c'}$ are orthogonal.  But it is easy to see that $c$ and $c'$ are independent realizations 
of the generic type $p_{0}$, whereby the group they generate is an elementary substructure of $\bar M$ isomorphic to $\F_{2}$, so by Claim 1, $q_{c}$ is orthogonal to $q_{c'}$.

The proof of Claim 2 finishes the proof of the Proposition.
\end{proof} 

\begin{remark} 
Note that the proof of the above Proposition yields that if the diagonal subgroup of $C(e_1)\times C(e_2)$ is not definable, then 
$C(e_1), C(e_2)$ are orthogonal. Moreover, as mentioned in the paragraph of the introduction after Lemma \ref{MainLemma}, the 
orthogonality of $C(e_1), C(e_2)$ is enough for proving that $T_{fg}$ has the DOP. 
\end{remark}

\section{Cyclic towers, and the proof of the main theorem}\label{CyclicTowers}
We start this section with the notion of an {\em amalgamated free product}, we refer the reader to \cite[Chapter IV]{LyndonSchupp} or to 
\cite[Section 4.4]{MagnusKarrassSolitar} for more details and motivation. We fix two groups $A,B$ a subgroup $C$ of $B$ and an embedding $f:C\to A$. 
Then the {\em amalgamated free product} $G:=A*_CB$ is the group $\langle A,B | c=f(c), \  c\in C\rangle$. Note that $G$ can be viewed as the free product 
$A*B$ quotiened by the normal subgroup containing $\{cf(c)^{-1} \ | \ c\in C\}$. This construction 
naturally arises in the context of algebraic topology for example in the Seifert - van Kampen theorem 
(see \cite[Section 1.2]{HatcherAlgTop}). 

For the rest of the section we fix a non abelian finitely generated free group $\F:=\langle \bar{e}\rangle$. For 
notational purposes, when an infinite cyclic group is denoted by a capital letter, say $C$, its generator will 
be denoted by the corresponding small letter $c$. 

\begin{definition}
Let $C$ be an infinite cyclic group. Then a cyclic tower over $\F$ is the amalgamated free product $\F*_C(C\oplus\Z)$ 
where $C$ embeds isomorphically onto a maximal abelian subgroup of $\F$.
\end{definition}

\begin{remark}
A cyclic tower $G:=\F*_C(C\oplus\Z)$ over $\F$ (with $f:C\hookrightarrow\F$) has an obvious group presentation. 
Suppose $f$ embeds $C$ isomorphically onto $C_{\F}(a)$.
Assume, without loss of generality, that $a\in\F$ is an element such that $C_{F}(a)=\langle a\rangle$, 
i.e. an element wihout proper roots. Then $G$ has the following presentation: $\langle \F, z \ | \ [z,a]\rangle$.      
\end{remark}

\begin{definition}
Let $G:=\F*_C(C\oplus\Z)$ be a cyclic tower over $\F$. Let $D$ be an infinite cyclic group and $f:C\oplus\Z\hookrightarrow C\oplus D$ 
be an injective morphism that is the identity on $C$, i.e. $f(c)=c$. Then the closure of $G$ with respect to $f$, $Cl_f(G)$, is the amalgamated free 
product $\F*_CB$ where $B$ is the group $C\oplus\Z\oplus D$ quotiented by the (normal) subgroup generated by $f(z)z^{-1}$.
\end{definition}

\begin{remark}
Let $G:=\F*_C(C\oplus\Z)$ be a cyclic tower over $\F$ with presentation $\langle \F, z \ | \ [z,a]\rangle$. 
Let $f:C\oplus\Z\hookrightarrow C\oplus D$ be an injective morphism with $f(c)=c$ and $f(z)=c^md^k$. 
Suppose $Cl_f(G)$ is the closure of $G$ with respect to $f$. Then: 
\begin{itemize}
 \item the injectivity of $f$ implies that $k$ must be non-zero;
 \item the closure of $G$ with respect to $f$, has an obvious presentation: $\langle \F,z,d \ | $ $\ [d,a] , a^md^kz^{-1}\rangle$;
 \item the group $G$ can be identified with the subgroup generated by $\F,z$ in its closure. 
\end{itemize}
\end{remark}

\begin{definition}
Let $G$ be a cyclic tower over $\F$ with presentation $\langle \F, z \ | \ [z,a]\rangle$. 
Then a test sequence with respect to $G$ is a sequence of morphisms $(h_n)_{n<\omega}:G\rightarrow\F$ 
with the following properties:
\begin{itemize}
 \item $h_n\upharpoonright\F=Id$ for all $n<\omega$;
 \item $h_n(z)=a^{k_n}$ with $(k_n)_{n<\omega}$ strictly increasing.
\end{itemize}
 
\end{definition}

\begin{definition}
Let $G$ be a cyclic tower over $\F$ and $Cl_f(G)$ be a closure (with respect to some $f$). We say that 
a test sequence $(h_n)_{n<\omega}:G\rightarrow \F$ extends to $Cl_f(G)$ if for all but finitely many $n$, $h_n$ extends  
to a morphism $h'_n:Cl_f(G)\rightarrow\F$.
\end{definition}

\begin{remark}
Let $G$ be a cyclic tower over $\F$ with presentation $\langle \F, z \ | \ [z,a]\rangle$. Then: 
\begin{itemize}
 \item a morphism from $G$ to $\F$ that is the identity on $\F$ is determined by the value it gives to $z$, which in turn must commute with $a$; 
 \item a test sequence with respect to $G$ can be identified with a sequence $(a^{k_n})_{n<\omega}$ of strictly increasing powers of $a$;
 \item if $f:C\oplus\Z\hookrightarrow C\oplus D$ is an injective morphism with $f(z)=c^md^k$ and $Cl_f(G)$ a closure of $G$ with respect to $f$. 
 Then a test sequence $(a^{k_n})_{n<\omega}$ 
 with respect to $G$ extends to $Cl_f(G)$ if and only if for all but finitely many $n$, $k_n\in m+k\Z$.
\end{itemize}
\end{remark}

The following theorem is a special case of Theorem 6.33 in \cite{FieldsByronSklinos}. 

\begin{theorem}\label{ExtFormalSolutions}
Let $G:=\langle \F, z \ | \ [z,a] \rangle$ be a cyclic tower over $\F$. 
Let $\phi(x,y)$ be a formula over $\F$
such that $\F\models\forall y\exists^{<\infty}x\phi(x,y)$. 

Suppose there exists a test sequence $(h_n)_{n<\omega}:G\rightarrow\F$ with respect to $G$ and a sequence $(b_n)_{n<\omega}$ of 
elements of $\F$ such that $\F\models \phi(b_n,h_n(z))$ for all $n$. 

Then there exists a closure $Cl_f(G):=\langle \F, z, d \ | \ [d,a], z^{-1}f(z)\rangle$ and a word  
$w=w(d,z,\bar{e})$ in $Cl_f(G)$ such that an infinite subsequence of $(h_n)_{n<\omega}$ extends to a 
sequence of morphisms $(h'_n)_{n<\omega}:Cl_f(G)\rightarrow\F$. Moreover, the extended sequence gives values 
to the couple $(w,z)$ that satisfy the formula $\phi(x,y)$, i.e. $\F\models \phi(h'_n(w),h'(z))$.
\end{theorem}

We can now prove as a corollary that the diagonal subgroup of $C_{\F}(e_1)\times C_{\F}(e_2)$, i.e. the cyclic group $\langle (e_1,e_2)\rangle$ 
is not definable in $\F$.

\begin{corollary}
The subgroup $\Gamma:=\langle (e_1,e_2)\rangle$ of $C_{\F}(e_1)\times C_{\F}(e_2)$ is not definable in $\F$. 
\end{corollary}
\begin{proof}
Suppose for a contradiction that the formula $\phi(x,y)$ over $\F$ defines $\Gamma$. We apply Theorem \ref{ExtFormalSolutions} 
to the cyclic tower $G:=\langle \F, z \ | \ [z,e_2]\rangle$ and the formula $\phi(x,y)$. 
We first see that $\F\models\forall y\exists^{<\infty} x \phi(x,y)$. 
Moreover, for the test sequence $(e_2^n)_{n<\omega}$ (with respect to $G$)  
there exists a sequence of elements of $\F$, namely $(e_1^n)_{n<\omega}$, such that $\F\models\phi(e_1^n,e_2^n)$ for all $n$. Thus, there exists a closure 
$Cl_{f}(G):=\langle \F,z,d \ | \ [d,e_2], z^{-1}f(z)\rangle$ of $G$ and a word $w=w(d,z,\bar{e})$ such that a subsequence 
$(e_2^{k_n})_{n<\omega}$ of $(e_2^n)_{n<\omega}$ extends to $Cl_{f}(G)$ and moreover if $(h'_n)_{n<\omega}:Cl_f(G)\rightarrow\F$ is the extended 
sequence, then $\F\models\phi(h'_n(w(d,z,\bar{e})), e_2^{k_n}))$ for all $n$. 

We observe that, since $d$ and $z$ commute with $e_2$ in $G$, for each $n$, $h'_n(d)$ and $h'_n(z)$ must be powers of $e_2$. 
On the other hand, in the word $w(d,z,\bar{e})$, the letter $e_1$ appears finitely many times. Since the only solution of $\phi(x,e_2^{k_n})$ 
is $e_1^{k_n}$, and by definition $h'_n(z)=e_2^{k_n}$, we must have that $h'_n(w(d,z,\bar{e})=w(e_2^{l_n},e_2^{k_n},\bar{e})=e_1^{k_n}$.  
But for $n$ large enough this is impossible. 

\end{proof}

Theorem \ref{Maintheorem} follows directly from  the Corollary above and Proposition \ref{PropReduction}.

\ \\ \\ \\ \\ \\ \\ \\
\begin{minipage}{0.3\textwidth}
Anand Pillay\\
University of Notre Dame, IN 46556, USA\\
apillay@nd.edu
\end{minipage}
\begin{minipage}{0.2\textwidth}
\
\end{minipage}
\begin{minipage}{0.3\textwidth}
Rizos Sklinos\\
Universit\'{e} Lyon 1, France\\
rizozs@gmail.com
\end{minipage}

\end{document}